\documentclass[11pt]{amsart}

\usepackage[T1]{fontenc}
\usepackage[utf8]{inputenc}
\usepackage{lmodern}
\usepackage{amsmath,amssymb,amsthm,mathtools,amscd}
\usepackage{enumitem}
\usepackage{microtype}
\usepackage{geometry}
\usepackage{hyperref}
\hypersetup{colorlinks=true,linkcolor=blue,citecolor=blue,urlcolor=blue}

\newtheorem{theorem}{Theorem}[section]
\newtheorem{proposition}[theorem]{Proposition}
\newtheorem{lemma}[theorem]{Lemma}
\newtheorem{corollary}[theorem]{Corollary}
\theoremstyle{definition}

\theoremstyle{remark}
\newtheorem{remark}[theorem]{Remark}

\newcommand{\Nil}{\operatorname{Nil}}
\newcommand{\Reg}{\operatorname{Reg}}
\newcommand{\Max}{\operatorname{Max}}
\newcommand{\Ann}{\operatorname{Ann}}
\newcommand{\fd}{\operatorname{fd}}
\newcommand{\pd}{\operatorname{pd}}
\newcommand{\Tor}{\operatorname{Tor}}
\newcommand{\Ext}{\operatorname{Ext}}
\newcommand{\wgd}{\operatorname{w.gl.dim}}
\newcommand{\gld}{\operatorname{gl.dim}}
\newcommand{\phiwgd}{\phi\text{-}\operatorname{w.gl.dim}}
\newcommand{\phigd}{\phi\text{-}\operatorname{gl.dim}}
\newcommand{\Frac}{\operatorname{Frac}}
\newcommand{\colim}{\varinjlim}

\title[On the $\phi$-Weak Global Dimensions and $\phi$-Global Dimensions]{On the $\phi$-Weak Global Dimensions and $\phi$-Global Dimensions}

\author{Xiaolei Zhang}
\address{School of Mathematics and Statistics, Tianshui Normal University, Tianshui 741000, China}
\email{zxlrghj@163.com}

\author{Wei Qi}
\address{School of Mathematics and Statistics, Tianshui Normal University, Tianshui 741000, China}
\email{qwrghj@126.com}


\begin{document}

\begin{abstract}
Let $R$ be a  $\phi$-ring and put $N=\Nil(R)$.  We determine the $\phi$-weak global dimension of every $\phi$-Pr\"ufer ring.  Its only possible values are $0$, $1$, and $\infty$: the value is $0$ when $R/N$ is a field, it is $1$ when $R$ is strong and $R/N$ is a nonfield Pr\"ufer domain, and it is infinite when $R$ is not strong.  No nilpotence assumption on $N$ is needed.  The proof is based on the canonical kernel $K=\ker(R\to\phi(R))$.  In the local non-strong case, $K$ is a nonzero torsion divisible module over the valuation domain $R/N$, and its self-Tor groups satisfy
\[
\Tor^R_{2n}(K,K)=0\quad(n\geq 1),
\qquad
\Tor^R_{2n+1}(K,K)\neq 0\quad(n\geq 0).
\]
It follows that $\fd_RK=\infty$, and a filtered-colimit argument transfers this obstruction to cyclic quotients by nonnil principal ideals.  We also prove the parallel trichotomy for the $\phi$-global dimension of $\phi$-Dedekind rings.

For strong $\phi$-rings we obtain a change-of-rings theorem without assuming that the quotient map $R\to R/N$ splits.  If $D=R/N$, $Q=\Frac(D)$ and $T=T(R)$ is the total quotient ring, then
\[
R\cong T\times_Q D.
\]
Using Ferrand's flat patching theorem and Milnor's projective patching theorem, we prove that every torsion $D$-module $M$, regarded as an $R$-module through $R\to D$, satisfies
\[
\fd_R(M)=\fd_D(M),
\qquad
\pd_R(M)=\pd_D(M).
\]
Consequently, for every strong $\phi$-ring,
\[
\phiwgd(R)=\wgd(R/N),
\qquad
\phigd(R)=\gld(R/N).
\]
\end{abstract}

\maketitle

\section{Introduction}

Throughout, rings are commutative with identity and modules are unitary.  A prime ideal $P$ is called \emph{divided} if $P\subsetneq xR$ for every $x\notin P$.  A ring $R$ is a $\phi$-ring if its nilradical $N=\Nil(R)$ is a divided prime ideal, and it is called a \emph{strong $\phi$-ring} if
\[
Z(R)=N.
\]
The basic theory of divided rings and $\phi$-rings was developed in \cite{Badawi1999,AndersonBadawi2004,Badawi2001}.  A $\phi$-ring is $\phi$-Pr\"ufer if every finitely generated nonnil ideal is $\phi$-invertible.  The fundamental structural characterization is
\[
R\text{ is }\phi\text{-Pr\"ufer}
\Longleftrightarrow
R_{\mathfrak m}\text{ is a }\phi\text{-chain ring for every }\mathfrak m\in\Max(R)
\Longleftrightarrow
R/N\text{ is a Pr\"ufer domain}.
\]
Similarly, $R$ is $\phi$-Dedekind if and only if $R/N$ is a Dedekind domain; see \cite{AndersonBadawi2004,AndersonBadawi2005,Zhao2018}.

Strongly $\phi$-flat modules and strongly nonnil-injective modules were introduced by Zhang, Xing and Qi \cite{ZhangXingQi}.  Their associated relative weak and ordinary global dimensions will be denoted by
\[
\phiwgd(R),\qquad \phigd(R).
\]
The invariant $\phiwgd$ considered here is the one defined through strongly $\phi$-flat resolutions.  In particular, for an NP-ring one has the cyclic formula
\begin{equation}\label{eq:intro-cyclic}
\begin{aligned}
\phiwgd(R)
&=\sup\{\fd_R(R/I): I\text{ is a nonnil ideal}\}\\
&=\sup\{\fd_R(R/I): I\text{ is finitely generated and nonnil}\}.
\end{aligned}
\end{equation}
Moreover,
\begin{equation}\label{eq:small-w}
\phiwgd(R)=0
\Longleftrightarrow
R/N\text{ is a field},
\end{equation}
and
\begin{equation}\label{eq:small-w-one}
\phiwgd(R)\leq 1
\Longleftrightarrow
R\text{ is a strong }\phi\text{-Pr\"ufer ring}.
\end{equation}
Thus the finite values $0$ and $1$ are understood.  The issue is whether a non-strong $\phi$-Pr\"ufer ring can have finite $\phi$-weak global dimension $2,3,\dots$.

For rings with square-zero nilradical, a trichotomy result of this type was obtained in recent work on the $\phi$-weak dimension of $\phi$-Pr\"ufer rings \cite{KimMahdouOubouhou}.  Our first main theorem removes every nilpotence assumption.

\begin{theorem}[Trichotomy for $\phi$-Pr\"ufer rings]\label{thm:intro-trichotomy}
Let $R$ be a $\phi$-Pr\"ufer ring and put $N=\Nil(R)$.  Then
\[
\phiwgd(R)=
\begin{cases}
0,& R/N\text{ is a field},\\
1,& Z(R)=N\text{ and }R/N\text{ is not a field},\\
\infty,& Z(R)\neq N.
\end{cases}
\]
In particular,
\[
\phiwgd(R)\in\{0,1,\infty\}.
\]
\end{theorem}

The main point is the non-strong case.  Let
\[
K=\ker(R\longrightarrow \phi(R))
\]
be the canonical kernel.  Dividedness of $N$ gives the elementary identity
\[
N=aN\qquad(a\notin N).
\]
From it we derive
\[
NK=K^2=0,
\qquad
aK=K\qquad(a\notin N).
\]
Consequently $K$ is naturally a torsion divisible module over the domain $D=R/N$.  In the local $\phi$-Pr\"ufer case, $D$ is a valuation domain.  We prove that for nonzero torsion divisible modules $U,V$ over an arbitrary valuation domain,
\[
\Tor^D_1(U,V)\neq 0,
\]
and this Tor module is again torsion divisible.  Combining this fact with a universal coefficient sequence yields an alternating self-Tor pattern for $K$ and hence $\fd_RK=\infty$.  Since $K$ is the filtered union of the annihilator ideals $(0:a)$ with $a\notin N$, finite $\phiwgd(R)$ would force finite flat dimension of $K$, a contradiction.

The second main result concerns the ordinary relative global dimension.

\begin{theorem}[Trichotomy for $\phi$-Dedekind rings]\label{thm:intro-dedekind}
Let $R$ be a $\phi$-Dedekind ring and put $N=\Nil(R)$.  Then
\[
\phigd(R)=
\begin{cases}
0,& R/N\text{ is a field},\\
1,& R\text{ is strong and }R/N\text{ is a nonfield Dedekind domain},\\
\infty,& R\text{ is not strong}.
\end{cases}
\]
Thus $\phigd(R)\in\{0,1,\infty\}$ for every $\phi$-Dedekind ring.
\end{theorem}

A third problem is the comparison with the reduced quotient.  Zhang--Xing--Qi asked whether, for a strong $\phi$-ring $R$,
\[
\phiwgd(R)=\wgd(R/N)
\qquad\text{and}\qquad
\phigd(R)=\gld(R/N).
\]
A direct extension-of-scalars argument is not available: the canonical map is the surjection $R\to D=R/N$, and there need not exist a ring homomorphism $D\to R$ splitting it.  Moreover, unless $N^2=0$, the nilradical itself is not naturally a $D$-module.  The correct replacement is a pullback construction.

Assume that $R$ is strong and put $S=R\setminus N$.  Then $S=\Reg(R)$, so $T=S^{-1}R=T(R)$.  Dividedness gives $N=sN$ for every $s\in S$, while strongness makes multiplication by $s$ injective on $N$.  Hence $N\to S^{-1}N$ is an isomorphism.  If $Q=\Frac(D)$, then $T/N\cong Q$, and we prove that the natural square
\[
\begin{CD}
R @>>> T\\
@VVV @VVV\\
D @>>> Q
\end{CD}
\]
is cartesian.  Thus
\begin{equation}\label{eq:intro-milnor}
R\cong T\times_QD.
\end{equation}
The map $T\to Q$ is surjective, so \eqref{eq:intro-milnor} is a Milnor square.  Flat modules can therefore be patched by Ferrand's theorem, while projective modules can be patched by Milnor's theorem.

This description gives a stronger result than the two global-dimension equalities.  If $M$ is any torsion $D$-module, then
\begin{equation}\label{eq:intro-module-equalities}
\fd_R(M)=\fd_D(M),
\qquad
\pd_R(M)=\pd_D(M).
\end{equation}
The proof has two ingredients.  First, multiplication by any element of $S$ is an automorphism of $N$, and this implies
\[
\Tor_i^R(N,M)=0\quad(i\geq0),
\qquad
\Tor_i^R(D,M)=0\quad(i\geq1).
\]
This gives the inequalities from $D$ to $R$.  For the reverse inequalities, a finite flat or projective $D$-resolution of $M$ becomes an exact complex of vector spaces after tensoring with $Q$, because $Q\otimes_DM=0$.  That exact $Q$-complex is lifted to a split exact complex of free $T$-modules and then patched degree by degree with the original $D$-resolution.  The resulting complex is a flat, respectively projective, $R$-resolution of the same length.  Applying \eqref{eq:intro-module-equalities} to $D/J$, $0\neq J\subset D$, yields
\begin{equation}\label{eq:intro-equalities}
\phiwgd(R)=\wgd(D),
\qquad
\phigd(R)=\gld(D).
\end{equation}

\section{Preliminaries}

Let $R$ be a $\phi$-ring and set $N=\Nil(R)$.  Since $N$ is divided, for every $a\notin N$ one has
\begin{equation}\label{eq:divided-basic}
N\subseteq aR.
\end{equation}
The canonical map from the total quotient ring into the localization at $N$ induces a homomorphism
\[
R\longrightarrow \phi(R).
\]
Its kernel is
\begin{equation}\label{eq:Kdef}
K=\{x\in R: xy=0\text{ for some }y\in Z(R)\setminus N\}.
\end{equation}
It is standard that
\begin{equation}\label{eq:Kstrong}
K\subseteq N,
\qquad
K=0\Longleftrightarrow Z(R)=N;
\end{equation}
see \cite[Proposition 2.1]{Zhao2018}.

An ideal is called \emph{nonnil} if it is not contained in $N$.  We will use the following characterizations.

\begin{theorem}\cite{Zhao2018}\label{thm:phi-prufer-char}
For a $\phi$-ring $R$ the following are equivalent.
\begin{enumerate}[label=\textup{(\arabic*)}]
\item $R$ is $\phi$-Pr\"ufer;
\item $R_{\mathfrak m}$ is a $\phi$-chain ring for every $\mathfrak m\in\Max(R)$;
\item $R/N$ is a Pr\"ufer domain;
\item $\phi(R)$ is a Pr\"ufer ring.
\end{enumerate}
\end{theorem}

In particular, if $(R,\mathfrak m)$ is local and $\phi$-Pr\"ufer, then $D=R/N$ is a valuation domain.

The following homological formulas  will be used repeatedly.

\begin{theorem}\cite{ZhangXingQi}\label{thm:ZXQ-weak}
Let $R$ be an NP-ring and let $n\geq 0$.  The following are equivalent.
\begin{enumerate}[label=\textup{(\arabic*)}]
\item $\phiwgd(R)\leq n$;
\item $\Tor^R_{n+1}(R/I,M)=0$ for every $R$-module $M$ and every nonnil ideal $I$;
\item $\fd_R(R/I)\leq n$ for every nonnil ideal $I$;
\item $\fd_R(R/I)\leq n$ for every finitely generated nonnil ideal $I$.
\end{enumerate}
Consequently,
\begin{equation}\label{eq:phi-w-sup}
\phiwgd(R)
=
\sup_{I\text{ nonnil}}\fd_R(R/I)
=
\sup_{I\text{ f.g. nonnil}}\fd_R(R/I).
\end{equation}
Moreover $\phiwgd(R)\leq \wgd(R)$.
\end{theorem}

\begin{theorem}\cite{ZhangXingQi}\label{thm:small-weak}
Let $R$ be a $\phi$-ring.
\begin{enumerate}[label=\textup{(\arabic*)}]
\item $\phiwgd(R)=0$ if and only if $R$ is $\phi$-von Neumann regular; equivalently, $R/N$ is a field.
\item $\phiwgd(R)\leq 1$ if and only if $R$ is a strong $\phi$-Pr\"ufer ring.
\end{enumerate}
\end{theorem}

The quotient by a nonnil ideal can always be written over the domain $D=R/N$.

\begin{proposition}\label{prop:nonnil-contains-N}
Let $R$ be a $\phi$-ring, $N=\Nil(R)$ and $D=R/N$.  Every nonnil ideal $I$ contains $N$.  Hence
\[
R/I\cong D/(I/N)
\]
for every nonnil ideal $I$, and
\begin{equation}\label{eq:phi-w-Dquotients}
\phiwgd(R)
=
\sup\{\fd_R(D/J):0\neq J\text{ is an ideal of }D\}.
\end{equation}
The same supremum may be taken over finitely generated nonzero ideals of $D$.
\end{proposition}

\begin{proof}
Let $I$ be nonnil and choose $a\in I\setminus N$.  By dividedness, $N\subseteq aR\subseteq I$.  Hence $J=I/N$ is a nonzero ideal of $D$ and $R/I\cong D/J$.

Conversely, let $0\neq J\subseteq D$, and let $I$ be its inverse image in $R$.  Then $N\subseteq I$ and $I$ is nonnil.  Thus \eqref{eq:phi-w-Dquotients} follows from Theorem \ref{thm:ZXQ-weak}.

It remains to justify the finite-generation assertion.  If $I$ is finitely generated then so is $I/N$.  Conversely, suppose
\[
J=(\overline a_1,\dots,\overline a_t)
\]
is a finitely generated nonzero ideal of $D$.  Choose an index $j$ with $a_j\notin N$.  The inverse image of $J$ is
\[
I=N+(a_1,\dots,a_t).
\]
Since $N\subseteq a_jR$ by dividedness, one gets
\[
I=(a_1,\dots,a_t),
\]
so $I$ is finitely generated.
\end{proof}

\begin{remark}
Formula \eqref{eq:phi-w-Dquotients} does not by itself identify $\phiwgd(R)$ with $\wgd(D)$: the flat dimension on the right is computed over $R$, not over $D$.  Section 9 gives a precise comparison.
\end{remark}

\section{Localization and principal reduction}

We first verify the behavior of nonnil elements under localization.

\begin{lemma}\label{lem:local-nonnil}
Let $\mathfrak m\in\Max(R)$ and let $a\notin N$.  If $a\in\mathfrak m$, then $a/1$ is not nilpotent in $R_{\mathfrak m}$.
\end{lemma}

\begin{proof}
Suppose $(a/1)^t=0$ for some $t\geq 1$.  Then there exists $s\notin\mathfrak m$ with
\[
sa^t=0.
\]
Since $0\in N$ and $N$ is prime, while $a^t\notin N$, it follows that $s\in N\subseteq\mathfrak m$, a contradiction.
\end{proof}

\begin{proposition}\label{prop:localization-phi-w}
For every $\phi$-ring $R$,
\[
\phiwgd(R)=\sup_{\mathfrak m\in\Max(R)}\phiwgd(R_{\mathfrak m}).
\]
\end{proposition}

\begin{proof}
Let $I$ be a finitely generated nonnil ideal.  By the local formula for flat dimension,
\[
\fd_R(R/I)=
\sup_{\mathfrak m\in\Max(R)}
\fd_{R_{\mathfrak m}}(R_{\mathfrak m}/I_{\mathfrak m}).
\]
If $I_{\mathfrak m}=R_{\mathfrak m}$, the corresponding term is zero.  Otherwise choose $a\in I\setminus N$.  Then $a\in\mathfrak m$, and Lemma \ref{lem:local-nonnil} shows that $a/1$ is nonnilpotent in $R_{\mathfrak m}$.  Hence $I_{\mathfrak m}$ is nonnil and
\[
\fd_{R_{\mathfrak m}}(R_{\mathfrak m}/I_{\mathfrak m})
\leq \phiwgd(R_{\mathfrak m}).
\]
Taking suprema and using Theorem \ref{thm:ZXQ-weak} gives
\[
\phiwgd(R)
\leq
\sup_{\mathfrak m}\phiwgd(R_{\mathfrak m}).
\]

Conversely, fix $\mathfrak m$ and let $J$ be a finitely generated nonnil ideal of $R_{\mathfrak m}$.  Write
\[
J=(x_1/s_1,\dots,x_t/s_t)
\qquad (s_i\notin\mathfrak m),
\]
and put $I=(x_1,\dots,x_t)$.  Then $J=I_{\mathfrak m}$.  If every $x_i$ belonged to $N$, then
\[
J\subseteq N_{\mathfrak m}=\Nil(R_{\mathfrak m}),
\]
contrary to the assumption that $J$ is nonnil.  Thus $I$ is nonnil.  Localization yields
\[
\fd_{R_{\mathfrak m}}(R_{\mathfrak m}/J)
\leq
\fd_R(R/I)
\leq
\phiwgd(R).
\]
Taking the supremum over $J$ proves the reverse inequality.
\end{proof}

\begin{corollary}\label{cor:principal-reduction}
Let $(R,\mathfrak m)$ be a local $\phi$-Pr\"ufer ring.  Every finitely generated nonnil ideal is principal, and
\[
\phiwgd(R)
=
\sup_{a\notin N}\fd_R(R/aR).
\]
\end{corollary}

\begin{proof}
By Theorem \ref{thm:phi-prufer-char}, $R$ is a $\phi$-chain ring.  Let
\[
I=(a_1,\dots,a_t)
\]
be nonnil.  Among the generators outside $N$, the corresponding principal ideals are linearly ordered.  Choose $a_j\notin N$ such that $a_jR$ contains every other nonnil generator.  Since $N\subseteq a_jR$, all nilpotent generators also belong to $a_jR$.  Thus $I=a_jR$.  The formula follows from Theorem \ref{thm:ZXQ-weak}.
\end{proof}

\section{The canonical kernel}

The next elementary identity drives the rest of the argument.

\begin{lemma}\label{lem:N=aN}
If $a\notin N$, then
\[
N=aN.
\]
\end{lemma}

\begin{proof}
Dividedness gives $N\subseteq aR$.  If $x\in N$, write $x=ar$.  Since $N$ is prime and $a\notin N$, one has $r\in N$.  Hence $N\subseteq aN$.  The reverse inclusion is immediate.
\end{proof}

\begin{proposition}\label{prop:K-identities}
Let $K$ be the ideal in \eqref{eq:Kdef}.  Then
\[
NK=0,
\qquad
K^2=0,
\qquad
aK=K\quad(a\notin N).
\]
Consequently $K$ is naturally a torsion divisible module over $D=R/N$.
\end{proposition}

\begin{proof}
Take $x\in K$.  By definition, there exists $s\in Z(R)\setminus N$ such that $sx=0$.  For $n\in N$, Lemma \ref{lem:N=aN} gives $n=sn_0$ for some $n_0\in N$.  Therefore
\[
nx=sn_0x=n_0(sx)=0.
\]
Thus $NK=0$.  Since $K\subseteq N$, this also gives $K^2=0$.

Now fix $a\notin N$.  Clearly $aK\subseteq K$.  Conversely, let $x\in K$.  By Lemma \ref{lem:N=aN}, write $x=ay$ with $y\in N$.  Choose $s\in Z(R)\setminus N$ with $sx=0$.  If $x=0$, there is nothing to prove.  Otherwise $y\neq 0$, and
\[
(sa)y=0.
\]
Since $N$ is prime, $sa\notin N$.  Moreover $sa$ is a zero-divisor because it annihilates the nonzero element $y$.  Hence $y\in K$ by \eqref{eq:Kdef}.  Thus $x\in aK$, proving $aK=K$.

Because $NK=0$, the $R$-action on $K$ factors through $D$.  Every nonzero element of $D$ is represented by some $a\notin N$, and multiplication by $a$ on $K$ is surjective; hence $K$ is divisible over $D$.  By definition of $K$, each element of $K$ is killed by some element of $R\setminus N$, whose image in $D$ is nonzero.  Thus $K$ is torsion.
\end{proof}

Put
\[
S=R\setminus N,
\qquad
L=N/K.
\]

\begin{lemma}\label{lem:L-local}
Multiplication by every $s\in S$ is an automorphism of $L$.  Hence $L$ is naturally an $S^{-1}R$-module, whereas
\[
S^{-1}K=0.
\]
\end{lemma}

\begin{proof}
Surjectivity on $L$ follows from $N=sN$.  Suppose
\[
s(x+K)=0
\]
for some $x\in N$.  Then $sx\in K$, so there exists $t\in Z(R)\setminus N$ with $tsx=0$.  Since $N$ is prime, $ts\notin N$.  Also $ts$ is a zero-divisor because $t$ is a zero-divisor.  Hence $x\in K$, proving injectivity.

Finally, every element of $K$ is annihilated by an element of $S$, so $S^{-1}K=0$.
\end{proof}

\begin{lemma}\label{lem:Tor-LK-zero}
For every $i\geq 0$,
\[
\Tor^R_i(L,K)=0.
\]
\end{lemma}

\begin{proof}
By Lemma \ref{lem:L-local}, multiplication by each $s\in S$ is an automorphism of $L$, and therefore also of $\Tor^R_i(L,K)$.  Thus $\Tor^R_i(L,K)$ is $S$-local.  On the other hand, localization of Tor gives
\[
S^{-1}\Tor^R_i(L,K)
\cong
\Tor^{S^{-1}R}_i(S^{-1}L,S^{-1}K)
=0.
\]
An $S$-local module whose localization is zero must itself be zero.
\end{proof}

We also record an elementary tensor vanishing result.

\begin{lemma}\label{lem:div-tors-tensor}
Let $D$ be a domain, let $U$ be a divisible $D$-module, and let $T$ be a torsion $D$-module.  Then
\[
U\otimes_D T=0.
\]
\end{lemma}

\begin{proof}
Given $u\otimes t$, choose $0\neq d\in D$ with $dt=0$.  Write $u=du_0$ by divisibility.  Then
\[
u\otimes t=u_0\otimes dt=0.
\]
\end{proof}

\begin{proposition}\label{prop:K-not-fg}
Let $(R,\mathfrak m)$ be a local non-strong $\phi$-ring.  Then
\[
K=\mathfrak mK.
\]
In particular, $K$ is not finitely generated.
\end{proposition}

\begin{proof}
Since $R$ is non-strong, choose $a\in Z(R)\setminus N$.  A zero-divisor is not a unit, so $a\in\mathfrak m$.  Proposition \ref{prop:K-identities} gives $aK=K$, hence
\[
K=aK\subseteq\mathfrak mK\subseteq K.
\]
Thus $K=\mathfrak mK$.  Since $K\neq 0$ by \eqref{eq:Kstrong}, Nakayama's lemma shows that $K$ cannot be finitely generated.
\end{proof}

\section{Torsion divisible modules over valuation domains}

The following theorem is the module-theoretic core of the self-Tor calculation.

\begin{theorem}\label{thm:Tor1-divisible}
Let $D$ be a valuation domain and let $U,V$ be nonzero torsion divisible $D$-modules.  Then
\[
\Tor^D_1(U,V)\neq 0.
\]
Moreover, $\Tor^D_1(U,V)$ is a torsion divisible $D$-module.
\end{theorem}

\begin{proof}
Choose $0\neq u_0\in U$ and $0\neq v_0\in V$.  Since $U$ and $V$ are torsion, there exist nonzero nonunits $\alpha,\beta\in D$ such that
\[
\alpha u_0=0,
\qquad
\beta v_0=0.
\]
Set $d=\alpha\beta$.  Then $d$ is a nonzero nonunit and
\[
du_0=dv_0=0.
\]
By divisibility, choose recursively elements $u_n\in U$ and $v_n\in V$ satisfying
\[
du_{n+1}=u_n,
\qquad
dv_{n+1}=v_n
\qquad(n\geq 0).
\]
Then for every $n\geq 0$,
\[
d^{n+1}u_n=0,
\qquad
d^nu_n=u_0\neq 0,
\]
and similarly
\[
d^{n+1}v_n=0,
\qquad
d^nv_n=v_0\neq 0.
\]
Fix $n\geq 2$ and put
\[
J=\Ann_D(u_n),
\qquad
I=\Ann_D(v_n).
\]
Thus
\[
d^{n+1}\in I\cap J,
\qquad
d^n\notin I\cup J.
\]
Since the ideals of a valuation domain are linearly ordered, comparison with the principal ideal $d^nD$ gives
\[
I\subseteq d^nD,
\qquad
J\subseteq d^nD.
\]
Hence
\[
IJ\subseteq d^{2n}D.
\]
We claim that
\begin{equation}\label{eq:d-not-IJ}
d^{n+1}\notin IJ.
\end{equation}
Indeed, if $d^{n+1}\in d^{2n}D$, then
\[
d^{n+1}=d^{2n}c
\]
for some $c\in D$.  Cancellation in the domain gives
\[
1=d^{n-1}c,
\]
which would make the nonunit $d$ invertible.  This proves \eqref{eq:d-not-IJ}.

The cyclic module $Dv_n\cong D/I$ embeds in $V$.  Every ideal of a valuation domain is flat, so tensoring this monomorphism with $J$ preserves injectivity:
\[
J\otimes_D D/I\longrightarrow J\otimes_DV.
\]
Under the standard isomorphism
\[
J\otimes_DD/I\cong J/IJ,
\]
the tensor $d^{n+1}\otimes v_n$ corresponds to the class of $d^{n+1}$ and is nonzero by \eqref{eq:d-not-IJ}.  Its image under the multiplication map
\[
J\otimes_DV\longrightarrow V
\]
is zero because $d^{n+1}v_n=0$.  From
\[
0\longrightarrow J\longrightarrow D\longrightarrow D/J\longrightarrow 0
\]
we conclude that
\[
\Tor^D_1(D/J,V)\neq 0.
\]
Now $Du_n\cong D/J$ embeds in $U$.  Since a valuation domain has weak global dimension at most one, the exact sequence
\[
0\longrightarrow D/J\longrightarrow U\longrightarrow U/Du_n\longrightarrow 0
\]
yields an injection
\[
\Tor^D_1(D/J,V)
\longrightarrow
\Tor^D_1(U,V).
\]
Thus $\Tor^D_1(U,V)\neq 0$.

It remains to prove divisibility and torsion.  Let $0\neq r\in D$.  Since $U$ is divisible,
\[
0\longrightarrow U[r]\longrightarrow U\xrightarrow{\,r\,}U\longrightarrow 0,
\qquad
U[r]=\{u\in U:ru=0\},
\]
is exact.  The module $U[r]$ is torsion.  By Lemma \ref{lem:div-tors-tensor}, applied after interchanging the two tensor factors,
\[
U[r]\otimes_DV=0.
\]
The long exact Tor sequence therefore shows that multiplication by $r$ on $\Tor^D_1(U,V)$ is surjective.  Hence $\Tor^D_1(U,V)$ is divisible.

Let $Q=\Frac(D)$.  Localization of Tor gives
\[
Q\otimes_D\Tor^D_1(U,V)
\cong
\Tor^Q_1(Q\otimes_DU,Q\otimes_DV)=0
\]
because $U$ and $V$ are torsion.  Thus $\Tor^D_1(U,V)$ is torsion.
\end{proof}

\begin{corollary}\label{cor:iterated-Tor}
Let $0\neq U$ be a torsion divisible module over a valuation domain $D$.  Define
\[
U_0=U,
\qquad
U_{n+1}=\Tor^D_1(U,U_n).
\]
Then every $U_n$ is a nonzero torsion divisible $D$-module.
\end{corollary}

\begin{proof}
Apply Theorem \ref{thm:Tor1-divisible} inductively.
\end{proof}

\section{A self-Tor computation for the canonical kernel}

Throughout this section $(R,\mathfrak m)$ is a local $\phi$-Pr\"ufer ring,
\[
N=\Nil(R),
\qquad
D=R/N,
\qquad
K=\ker(R\to\phi(R)),
\]
and $K\neq 0$.  Then $D$ is a valuation domain.

We begin with the universal coefficient sequence needed below.

\begin{lemma}[Universal coefficient lemma]\label{lem:UCT}
Let $A$ be a ring with $\wgd(A)\leq 1$, let $C_\bullet$ be a chain complex of flat $A$-modules, and let $M$ be an $A$-module.  For every integer $n$ there is a natural short exact sequence
\begin{equation}\label{eq:UCT}
0\longrightarrow H_n(C_\bullet)\otimes_AM
\longrightarrow H_n(C_\bullet\otimes_AM)
\longrightarrow \Tor^A_1(H_{n-1}(C_\bullet),M)
\longrightarrow 0.
\end{equation}
\end{lemma}

\begin{proof}
Write
\[
Z_n=\ker(d_n:C_n\to C_{n-1}),
\qquad
B_n=\operatorname{Im}(d_{n+1}:C_{n+1}\to C_n).
\]
Since $\wgd(A)\leq 1$, every submodule of a flat $A$-module is flat.  Hence each $Z_n$ and $B_n$ is flat.  Therefore
\[
0\longrightarrow Z_n\longrightarrow C_n\longrightarrow B_{n-1}\longrightarrow 0
\]
remains exact after tensoring with $M$:
\begin{equation}\label{eq:tensor-ZCB}
0\longrightarrow Z_n\otimes_AM
\longrightarrow C_n\otimes_AM
\longrightarrow B_{n-1}\otimes_AM
\longrightarrow 0.
\end{equation}
The boundaries in $C_\bullet\otimes_AM$ are precisely the image of
\[
B_n\otimes_AM\longrightarrow Z_n\otimes_AM.
\]
Put $Q_{n-1}=C_{n-1}/B_{n-1}$.  From
\[
0\longrightarrow B_{n-1}\longrightarrow C_{n-1}\longrightarrow Q_{n-1}\longrightarrow 0
\]
one gets
\[
\ker(B_{n-1}\otimes_AM\to C_{n-1}\otimes_AM)
\cong
\Tor^A_1(Q_{n-1},M).
\]
The exact sequence
\[
0\longrightarrow H_{n-1}(C_\bullet)
\longrightarrow Q_{n-1}
\longrightarrow B_{n-2}
\longrightarrow 0
\]
and the flatness of $B_{n-2}$ yield
\[
\Tor^A_1(Q_{n-1},M)
\cong
\Tor^A_1(H_{n-1}(C_\bullet),M).
\]
Using \eqref{eq:tensor-ZCB} and then quotienting cycles by boundaries gives \eqref{eq:UCT}.
\end{proof}

Let $P_\bullet\to K$ be an $R$-projective resolution and put
\[
C_\bullet=D\otimes_RP_\bullet.
\]
Set
\[
T_i=\Tor^R_i(K,K).
\]

\begin{lemma}\label{lem:C-homology}
With the notation above,
\[
H_0(C_\bullet)\cong K,
\qquad
H_1(C_\bullet)=0,
\]
and for every $q\geq 2$,
\[
H_q(C_\bullet)\cong T_{q-1}.
\]
Moreover,
\[
H_n(C_\bullet\otimes_DK)\cong T_n
\qquad(n\geq 0).
\]
\end{lemma}

\begin{proof}
By Proposition \ref{prop:K-identities}, $K$ is a torsion divisible $D$-module and $NK=0$.  Thus
\[
K\otimes_RK\cong K\otimes_DK=0
\]
by Lemma \ref{lem:div-tors-tensor}.

Consider
\[
0\longrightarrow K\longrightarrow N\longrightarrow L=N/K\longrightarrow 0.
\]
Lemma \ref{lem:Tor-LK-zero} and the preceding tensor vanishing give
\[
N\otimes_RK=0
\]
and, for every $i\geq 1$,
\[
\Tor^R_i(N,K)\cong \Tor^R_i(K,K)=T_i.
\]
Applying $-\otimes_RK$ to
\[
0\longrightarrow N\longrightarrow R\longrightarrow D\longrightarrow 0
\]
therefore gives
\[
\Tor^R_1(D,K)=0
\]
and
\[
\Tor^R_q(D,K)\cong T_{q-1}
\qquad(q\geq 2).
\]
Since
\[
H_q(C_\bullet)=\Tor^R_q(D,K),
\]
and $D\otimes_RK\cong K$, the first assertions follow.

Finally, associativity of tensor products gives a natural isomorphism of complexes
\[
C_\bullet\otimes_DK
\cong
P_\bullet\otimes_RK.
\]
Therefore
\[
H_n(C_\bullet\otimes_DK)
\cong
\Tor^R_n(K,K)=T_n.
\]
\end{proof}

\begin{theorem}\label{thm:self-Tor}
Let $(R,\mathfrak m)$ be a local non-strong $\phi$-Pr\"ufer ring.  Then
\[
\Tor^R_{2n}(K,K)=0
\qquad(n\geq 1),
\]
whereas
\[
\Tor^R_{2n+1}(K,K)\neq 0
\qquad(n\geq 0).
\]
Every odd self-Tor group is a nonzero torsion divisible $D$-module.  In particular,
\[
\fd_RK=\infty.
\]
\end{theorem}

\begin{proof}
Let $P_\bullet\to K$ and $C_\bullet=D\otimes_RP_\bullet$ be as above.  Since $D$ is a valuation domain,
\[
\wgd(D)\leq 1,
\]
and each $C_i$ is a projective, hence flat, $D$-module.  Apply Lemma \ref{lem:UCT} with $A=D$ and $M=K$.  By Lemma \ref{lem:C-homology}, for every $n\geq 1$ there is an exact sequence
\begin{equation}\label{eq:recurrence}
0\longrightarrow H_n(C_\bullet)\otimes_DK
\longrightarrow T_n
\longrightarrow
\Tor^D_1(H_{n-1}(C_\bullet),K)
\longrightarrow 0.
\end{equation}

For $n=1$, Lemma \ref{lem:C-homology} gives
\[
H_1(C_\bullet)=0,
\qquad
H_0(C_\bullet)=K.
\]
Hence
\[
T_1\cong\Tor^D_1(K,K),
\]
which is nonzero torsion divisible by Theorem \ref{thm:Tor1-divisible}.

For $n=2$,
\[
H_2(C_\bullet)\cong T_1,
\qquad
H_1(C_\bullet)=0.
\]
Since $T_1$ is torsion and $K$ is divisible, Lemma \ref{lem:div-tors-tensor} gives
\[
T_1\otimes_DK=0.
\]
Thus \eqref{eq:recurrence} yields $T_2=0$.

Assume inductively that for some $m\geq 1$,
\[
T_{2m}=0
\]
and $T_{2m-1}$ is nonzero torsion divisible.  For $n=2m+1$,
\[
H_{2m+1}(C_\bullet)\cong T_{2m}=0,
\qquad
H_{2m}(C_\bullet)\cong T_{2m-1}.
\]
Hence
\[
T_{2m+1}
\cong
\Tor^D_1(T_{2m-1},K),
\]
which is again nonzero torsion divisible by Theorem \ref{thm:Tor1-divisible}.

For $n=2m+2$,
\[
H_{2m+2}(C_\bullet)\cong T_{2m+1},
\qquad
H_{2m+1}(C_\bullet)\cong T_{2m}=0.
\]
The first term in \eqref{eq:recurrence} vanishes by Lemma \ref{lem:div-tors-tensor}, because $T_{2m+1}$ is torsion and $K$ is divisible, while the last term is zero.  Thus $T_{2m+2}=0$.  Induction proves the asserted pattern.

Since nonzero Tor groups occur in arbitrarily large degrees, $\fd_RK=\infty$.
\end{proof}

\section{The trichotomy theorem}

For $a\notin N$, put
\[
A_a=(0:_Ra).
\]

\begin{lemma}\label{lem:K-filtered}
One has
\[
K=\bigcup_{a\notin N}A_a,
\]
and this union is filtered.  Moreover, if
\[
\fd_R(R/aR)\leq d<\infty
\]
with $d\geq 2$, then
\[
\fd_RA_a\leq d-2.
\]
\end{lemma}

\begin{proof}
The equality follows directly from \eqref{eq:Kdef}.  If $a,b\notin N$, then $ab\notin N$ because $N$ is prime, and
\[
A_a+A_b\subseteq A_{ab}.
\]
Hence the union is filtered.

The exact sequences
\[
0\longrightarrow A_a\longrightarrow R\xrightarrow{\,a\,}aR\longrightarrow 0
\]
and
\[
0\longrightarrow aR\longrightarrow R\longrightarrow R/aR\longrightarrow 0
\]
give, for every $R$-module $M$ and every $i\geq 1$,
\[
\Tor^R_i(A_a,M)
\cong
\Tor^R_{i+2}(R/aR,M).
\]
Thus $\fd_RA_a\leq d-2$ whenever $\fd_R(R/aR)\leq d$ and $d\geq 2$.
\end{proof}

\begin{theorem}\label{thm:local-infinite}
Let $(R,\mathfrak m)$ be a local non-strong $\phi$-Pr\"ufer ring.  Then
\[
\phiwgd(R)=\infty.
\]
\end{theorem}

\begin{proof}
Assume $\phiwgd(R)=d<\infty$.  By Theorem \ref{thm:small-weak}, the non-strong hypothesis forces $d\geq 2$.  By Corollary \ref{cor:principal-reduction},
\[
\fd_R(R/aR)\leq d
\qquad(a\notin N).
\]
Hence Lemma \ref{lem:K-filtered} gives
\[
\fd_RA_a\leq d-2
\qquad(a\notin N).
\]
Since $K$ is the filtered colimit of the $A_a$ and Tor commutes with filtered colimits in either variable over a commutative ring, for every $R$-module $M$ one gets
\[
\Tor^R_i(K,M)
\cong
\colim_{a\notin N}\Tor^R_i(A_a,M)=0
\qquad(i>d-2).
\]
Therefore $\fd_RK\leq d-2$, contradicting Theorem \ref{thm:self-Tor}.  Hence $\phiwgd(R)=\infty$.
\end{proof}

There is also a useful low-degree obstruction valid without the Pr\"ufer hypothesis.

\begin{proposition}\label{prop:low-degree}
Let $R$ be a non-strong $\phi$-ring.  Then there exists $a\in Z(R)\setminus N$ such that, with
\[
A=(0:_Ra),
\qquad
C=R/aR,
\]
one has
\[
\Tor^R_2(C,C)\cong A,
\qquad
\Tor^R_3(C,C)\cong A.
\]
In particular,
\[
\phiwgd(R)\geq 3.
\]
\end{proposition}

\begin{proof}
Choose $a\in Z(R)\setminus N$ and put $A=(0:a)$.  Then $A\neq 0$.  Since $A\subseteq K$, Proposition \ref{prop:K-identities} gives
\[
A^2=0.
\]
Also, dividedness gives
\[
A\subseteq N\subseteq aR,
\]
and by definition $aA=0$.

From
\[
0\longrightarrow aR\longrightarrow R\longrightarrow C\longrightarrow 0
\]
one gets
\[
\Tor^R_2(C,C)\cong\Tor^R_1(aR,C).
\]
Tensoring
\[
0\longrightarrow A\longrightarrow R\longrightarrow aR\longrightarrow 0
\]
with $C$ yields
\[
\Tor^R_1(aR,C)
\cong
\ker(A\otimes_RC\longrightarrow C).
\]
Now
\[
A\otimes_RC\cong A/aA\cong A,
\]
and the map to $C$ is zero because $A\subseteq aR$.  Hence
\[
\Tor^R_2(C,C)\cong A.
\]

Similarly, dimension shifting gives
\[
\Tor^R_3(C,C)
\cong
\Tor^R_1(A,C)
\cong
\Tor^R_1(C,A).
\]
Since $aR\cong R/A$,
\[
aR\otimes_RA\cong A/A^2\cong A.
\]
The multiplication map $aR\otimes_RA\to A$ is zero because $aA=0$.  The long exact Tor sequence associated to
\[
0\longrightarrow aR\longrightarrow R\longrightarrow C\longrightarrow 0
\]
therefore yields
\[
\Tor^R_1(C,A)\cong A.
\]
Thus $\Tor^R_3(C,C)\cong A\neq 0$.

The ideal $aR$ is finitely generated and nonnil.  Formula \eqref{eq:phi-w-sup} gives $\phiwgd(R)\geq 3$.
\end{proof}

\begin{corollary}\label{cor:no-two}
No $\phi$-Pr\"ufer ring has $\phi$-weak global dimension equal to $2$.
\end{corollary}

\begin{proof}
If $R$ is strong, Theorem \ref{thm:small-weak} gives $\phiwgd(R)\leq 1$.  If $R$ is non-strong, Proposition \ref{prop:low-degree} gives $\phiwgd(R)\geq 3$.
\end{proof}

\begin{lemma}\label{lem:localize-nonstrong}
If $R$ is a non-strong $\phi$-ring, then $R_{\mathfrak m}$ is non-strong for some $\mathfrak m\in\Max(R)$.
\end{lemma}

\begin{proof}
By \eqref{eq:Kstrong}, $K\neq 0$.  Choose $0\neq x\in K$.  Pick a maximal ideal $\mathfrak m$ containing $\Ann_R(x)$; then $x/1\neq 0$ in $R_{\mathfrak m}$.  There exists $a\in Z(R)\setminus N$ with $ax=0$.  Necessarily $a\in\mathfrak m$, for otherwise $a/1$ would be a unit and would force $x/1=0$.  By Lemma \ref{lem:local-nonnil}, $a/1$ is nonnilpotent.  Since
\[
(a/1)(x/1)=0
\]
with $x/1\neq 0$, the element $a/1$ is a zero-divisor outside $\Nil(R_{\mathfrak m})=N_{\mathfrak m}$.  Thus $R_{\mathfrak m}$ is non-strong.
\end{proof}

\begin{theorem}\label{thm:global-infinite}
Let $R$ be a non-strong $\phi$-Pr\"ufer ring.  Then
\[
\phiwgd(R)=\infty.
\]
\end{theorem}

\begin{proof}
Choose $\mathfrak m$ as in Lemma \ref{lem:localize-nonstrong}.  Since $R$ is $\phi$-Pr\"ufer, so is $R_{\mathfrak m}$.  Theorem \ref{thm:local-infinite} yields
\[
\phiwgd(R_{\mathfrak m})=\infty.
\]
Proposition \ref{prop:localization-phi-w} now gives $\phiwgd(R)=\infty$.
\end{proof}

We can now prove the first main theorem.

\begin{theorem}\label{thm:trichotomy-final}
Let $R$ be a $\phi$-Pr\"ufer ring and put $N=\Nil(R)$.  Then
\[
\phiwgd(R)=
\begin{cases}
0,& R/N\text{ is a field},\\
1,& Z(R)=N\text{ and }R/N\text{ is not a field},\\
\infty,& Z(R)\neq N.
\end{cases}
\]
\end{theorem}

\begin{proof}
If $Z(R)\neq N$, apply Theorem \ref{thm:global-infinite}.  Suppose $Z(R)=N$.  Then $R$ is strong.  Since $R$ is $\phi$-Pr\"ufer, Theorem \ref{thm:small-weak} gives $\phiwgd(R)\leq 1$.  The same theorem gives $\phiwgd(R)=0$ exactly when $R/N$ is a field.  Thus the remaining strong case has $\phiwgd(R)=1$.
\end{proof}

\begin{corollary}\label{cor:finite-iff-strong}
For a $\phi$-Pr\"ufer ring $R$, the following are equivalent:
\[
\phiwgd(R)<\infty,
\qquad
\phiwgd(R)\leq 1,
\qquad
R\text{ is strong}.
\]
Moreover, every non-strong $\phi$-Pr\"ufer ring has
\[
\wgd(R)=\infty.
\]
\end{corollary}

\begin{proof}
The equivalences follow from Theorem \ref{thm:trichotomy-final} and Theorem \ref{thm:small-weak}.  Since $\phiwgd(R)\leq\wgd(R)$, the last assertion follows from Theorem \ref{thm:global-infinite}.
\end{proof}

\begin{corollary}\label{cor:Kcriterion}
Let $R$ be a $\phi$-Pr\"ufer ring and let $K=\ker(R\to\phi(R))$.  Then
\[
\phiwgd(R)<\infty
\Longleftrightarrow
K=0.
\]
If these conditions hold, then $\phiwgd(R)\leq 1$.
\end{corollary}

\begin{proof}
By \eqref{eq:Kstrong}, $K=0$ is equivalent to strongness.  Apply Corollary \ref{cor:finite-iff-strong}.
\end{proof}

\section{The \texorpdfstring{$\phi$}{phi}-global dimension of \texorpdfstring{$\phi$}{phi}-Dedekind rings}

For an $R$-module $M$, let $\phi\text{-}\operatorname{id}_R(M)$ denote its dimension with respect to strongly nonnil-injective modules.  The $\phi$-global dimension is
\[
\phigd(R)
=
\sup\{\phi\text{-}\operatorname{id}_R(M):M\text{ is an }R\text{-module}\}.
\]
By \cite{ZhangXingQi},
\begin{equation}\label{eq:phi-global-cyclic}
\phigd(R)
=
\sup\{\pd_R(R/I):I\text{ is a nonnil ideal of }R\}.
\end{equation}

\begin{lemma}\label{lem:wleqg}
For every $\phi$-ring $R$,
\[
\phiwgd(R)\leq \phigd(R).
\]
\end{lemma}

\begin{proof}
For every nonnil ideal $I$,
\[
\fd_R(R/I)\leq\pd_R(R/I).
\]
Taking suprema in \eqref{eq:phi-w-sup} and \eqref{eq:phi-global-cyclic} gives the result.
\end{proof}

We also use the small-dimensional characterization from \cite{ZhangXingQi}.

\begin{theorem}\label{thm:small-global}
Let $R$ be a $\phi$-ring.
\begin{enumerate}[label=\textup{(\arabic*)}]
\item $\phigd(R)=0$ if and only if $R$ is $\phi$-von Neumann regular; equivalently, $R/N$ is a field.
\item $\phigd(R)\leq 1$ if and only if $R$ is a strong $\phi$-Dedekind ring.
\end{enumerate}
\end{theorem}

Recall that
\begin{equation}\label{eq:phi-Dedekind-char}
R\text{ is }\phi\text{-Dedekind}
\Longleftrightarrow
R/N\text{ is a Dedekind domain};
\end{equation}
see \cite{AndersonBadawi2005}.  In particular, every $\phi$-Dedekind ring is $\phi$-Pr\"ufer.

\begin{theorem}\label{thm:phi-global-trichotomy}
Let $R$ be a $\phi$-Dedekind ring and put $N=\Nil(R)$.  Then
\[
\phigd(R)=
\begin{cases}
0,& R/N\text{ is a field},\\
1,& R\text{ is strong and }R/N\text{ is a nonfield Dedekind domain},\\
\infty,& R\text{ is not strong}.
\end{cases}
\]
\end{theorem}

\begin{proof}
If $R/N$ is a field, Theorem \ref{thm:small-global}(1) gives $\phigd(R)=0$.

Suppose $R$ is strong and $R/N$ is not a field.  By \eqref{eq:phi-Dedekind-char}, $R/N$ is a nonfield Dedekind domain.  Theorem \ref{thm:small-global}(2) gives $\phigd(R)\leq 1$, while Theorem \ref{thm:small-global}(1) excludes the value $0$.  Hence $\phigd(R)=1$.

Finally suppose $R$ is not strong.  Since every $\phi$-Dedekind ring is $\phi$-Pr\"ufer, Theorem \ref{thm:global-infinite} gives
\[
\phiwgd(R)=\infty.
\]
Lemma \ref{lem:wleqg} forces $\phigd(R)=\infty$.
\end{proof}

\begin{corollary}\label{cor:global-finite-strong}
For a $\phi$-Dedekind ring $R$, the following are equivalent:
\[
\phigd(R)<\infty,
\qquad
\phigd(R)\leq 1,
\qquad
R\text{ is strong}.
\]
\end{corollary}

\section{Strong \texorpdfstring{$\phi$}{phi}-rings as Milnor pullbacks}\label{sec:milnor}

Throughout this section $R$ is a strong $\phi$-ring.  We write
\[
N=\Nil(R),\qquad D=R/N,\qquad S=R\setminus N,
\]
and denote the quotient field of the domain $D$ by $Q$.  Since $R$ is strong,
\[
S=\Reg(R),
\]
so the localization $S^{-1}R$ is the total quotient ring $T(R)$.  We abbreviate it by $T$.

\subsection{The pullback square}

The first observation is the strong form of Lemma \ref{lem:N=aN}.

\begin{lemma}\label{lem:strong-N-auto}
For every $s\in S$, multiplication by $s$ is an automorphism of $N$.
Consequently, the localization map
\[
N\longrightarrow S^{-1}N
\]
is an isomorphism.
\end{lemma}

\begin{proof}
Lemma \ref{lem:N=aN} gives $N=sN$, so multiplication by $s$ is surjective on $N$.  Since $R$ is strong, $s\notin N=Z(R)$, and therefore $s$ is a non-zero-divisor on $R$.  Its restriction to $N$ is injective.  Thus multiplication by $s$ is an automorphism of $N$.

For the localization statement, every fraction $n/s\in S^{-1}N$ has a unique representative in $N$: choose the unique $m\in N$ with $sm=n$, and then $n/s=m/1$.  Injectivity of $N\to S^{-1}N$ follows because every element of $S$ is regular on $N$.
\end{proof}

We identify $N$ with $S^{-1}N$ through this lemma.

\begin{proposition}\label{prop:TmodN=Q}
There is a natural exact sequence
\[
0\longrightarrow N\longrightarrow T\longrightarrow Q\longrightarrow0.
\]
In particular,
\[
T/N\cong Q.
\]
\end{proposition}

\begin{proof}
Localizing the exact sequence
\[
0\longrightarrow N\longrightarrow R\longrightarrow D\longrightarrow0
\]
at $S$ gives
\[
0\longrightarrow S^{-1}N\longrightarrow S^{-1}R
   \longrightarrow S^{-1}D\longrightarrow0.
\]
By Lemma \ref{lem:strong-N-auto}, $S^{-1}N=N$, while $S^{-1}R=T$.  The image of $S$ in $D$ is precisely $D\setminus\{0\}$, hence
\[
S^{-1}D=\Frac(D)=Q.
\]
The assertion follows.
\end{proof}

\begin{theorem}[Canonical Milnor square]\label{thm:canonical-milnor}
For every strong $\phi$-ring $R$, the square
\begin{equation}\label{eq:canonical-square}
\begin{CD}
R @>{\iota}>> T\\
@V{\pi}VV @VV{\rho}V\\
D @>{j}>> Q
\end{CD}
\end{equation}
is cartesian.  Equivalently,
\begin{equation}\label{eq:R-pullback}
R\cong T\times_QD.
\end{equation}
Moreover, $\rho:T\to Q$ is surjective.
\end{theorem}

\begin{proof}
The maps in \eqref{eq:canonical-square} are the localization map $\iota$, the quotient map $\pi$, the map $\rho$ from Proposition \ref{prop:TmodN=Q}, and the natural inclusion $j:D\hookrightarrow Q$.  The square commutes.

Since $S$ consists of non-zero-divisors, $\iota:R\to T$ is injective.  Hence the induced map
\[
\Psi:R\longrightarrow T\times_QD,
\qquad
r\longmapsto(r/1,r+N),
\]
is injective.

For surjectivity, take $(t,d)\in T\times_QD$.  Choose $r\in R$ with $\pi(r)=d$.  The equality
\[
\rho(t)=j(d)=\rho(r/1)
\]
in $Q$ gives
\[
t-r/1\in\ker(\rho)=N
\]
by Proposition \ref{prop:TmodN=Q}.  Under the identification $N=S^{-1}N$ of Lemma \ref{lem:strong-N-auto}, there exists $n\in N\subseteq R$ such that
\[
t-r/1=n/1.
\]
Thus $t=(r+n)/1$.  Since $n\in N$,
\[
\pi(r+n)=\pi(r)=d,
\]
and therefore $(t,d)=\Psi(r+n)$.  This proves \eqref{eq:R-pullback}.  The surjectivity of $\rho$ was already obtained in Proposition \ref{prop:TmodN=Q}.
\end{proof}

\begin{corollary}\label{cor:MV-ring-sequence}
There is a short exact sequence of $R$-modules
\begin{equation}\label{eq:ring-MV}
0\longrightarrow R
\xrightarrow{\ (\iota,\pi)\ }
T\oplus D
\xrightarrow{\ \rho-j\ }
Q
\longrightarrow0.
\end{equation}
\end{corollary}

\begin{proof}
The kernel of $(t,d)\mapsto\rho(t)-j(d)$ is the fiber product $T\times_QD$, which is $R$ by Theorem \ref{thm:canonical-milnor}.  The last map is surjective because $\rho:T\to Q$ is surjective.
\end{proof}

\subsection{Tor-acyclicity of torsion quotient modules}

The next lemma is the homological input needed for the inequality in one direction.

\begin{lemma}\label{lem:Tor-N-torsion-zero}
Let $M$ be a torsion $D$-module, regarded as an $R$-module via $\pi:R\to D$.  Then
\begin{equation}\label{eq:Tor-N-torsion-zero}
\Tor_i^R(N,M)=0
\qquad(i\geq0).
\end{equation}
Consequently,
\begin{equation}\label{eq:Tor-D-torsion-zero}
\Tor_i^R(D,M)=0
\qquad(i\geq1).
\end{equation}
\end{lemma}

\begin{proof}
Put
\[
X_i=\Tor_i^R(N,M).
\]
For every $s\in S$, multiplication by $s$ on $N$ is an automorphism by Lemma \ref{lem:strong-N-auto}.  Functoriality of Tor therefore shows that multiplication by $s$ is an automorphism of $X_i$.  Hence $X_i$ is $S$-local.

Since $M$ is torsion over the domain $D$, every $m\in M$ is annihilated by some nonzero element of $D$.  Lifting such an element to $S$ shows that
\[
S^{-1}M=0.
\]
Localization of Tor gives
\[
S^{-1}X_i
\cong
\Tor_i^{T}(S^{-1}N,S^{-1}M)=0.
\]
Because $X_i$ is already $S$-local, the canonical map $X_i\to S^{-1}X_i$ is an isomorphism.  Thus $X_i=0$ for every $i\geq0$, proving \eqref{eq:Tor-N-torsion-zero}.

Apply $-\otimes_RM$ to
\[
0\longrightarrow N\longrightarrow R\longrightarrow D\longrightarrow0.
\]
For $i\geq2$ the long exact Tor sequence gives
\[
\Tor_i^R(D,M)\cong\Tor_{i-1}^R(N,M)=0,
\]
and for $i=1$ it gives an injection
\[
\Tor_1^R(D,M)\hookrightarrow N\otimes_RM=0.
\]
This proves \eqref{eq:Tor-D-torsion-zero}.
\end{proof}

\begin{remark}\label{rem:annihilated-version}
If $M$ is annihilated by a single element $s\in S$, the first part of the preceding proof can be shortened: multiplication by $s$ on $\Tor_i^R(N,M)$ is simultaneously an automorphism, because it is so on $N$, and the zero map, because $sM=0$.  The localization argument is needed only to treat an arbitrary torsion $D$-module at once.
\end{remark}

\subsection{The patching theorems}

We record the precise forms of the two patching results that will be used.  Let
\begin{equation}\label{eq:general-milnor}
\begin{CD}
A @>>> A_1\\
@VVV @VV{\alpha_1}V\\
A_2 @>{\alpha_2}>> C
\end{CD}
\end{equation}
be a cartesian square of rings and assume that $\alpha_1:A_1\to C$ is surjective.  Given an $A_1$-module $M_1$, an $A_2$-module $M_2$, and a $C$-linear isomorphism
\[
\theta:C\otimes_{A_1}M_1
\xrightarrow{\ \sim\ }
C\otimes_{A_2}M_2,
\]
define
\begin{equation}\label{eq:patched-module}
P(M_1,M_2;\theta)
=
\left\{(x_1,x_2)\in M_1\oplus M_2:
\theta(1\otimes x_1)=1\otimes x_2\right\}.
\end{equation}

\begin{theorem}[Milnor--Ferrand patching]\label{thm:patching}
In the situation of \eqref{eq:general-milnor}:
\begin{enumerate}[label=\textup{(\arabic*)}]
\item if $M_1$ and $M_2$ are flat over $A_1$ and $A_2$, respectively, then $P(M_1,M_2;\theta)$ is flat over $A$;
\item if $M_1$ and $M_2$ are projective over $A_1$ and $A_2$, respectively, then $P(M_1,M_2;\theta)$ is projective over $A$.
\end{enumerate}
In both cases the canonical base-change maps recover $M_1$ and $M_2$.
\end{theorem}

\begin{proof}
The flat assertion, including the base-change statement, is Ferrand's theorem on modules over a fiber product; see \cite[Theorem 2.2(iv)]{Ferrand2003}.  The projective assertion is Milnor's patching theorem.  We use the version for arbitrary projective modules, equivalently the equivalence between the category of projective $A$-modules and the category of projective gluing triples; see \cite[Section 2]{Milnor1971} and \cite[Theorem 2.3]{ChenLe2020}.  The finite-generation hypothesis is needed only when one wants to preserve finite generation, not for projectivity itself.
\end{proof}

The flat part of Theorem \ref{thm:patching} is important here because a finite flat resolution need not be a projective resolution.  The projective part is used independently for the ordinary global dimension.

\begin{lemma}\label{lem:lift-split-complex}
Let $T\to Q$ be a surjective homomorphism onto a field.  Every bounded exact complex $V_\bullet$ of $Q$-vector spaces is isomorphic to
\[
Q\otimes_TG_\bullet
\]
for some bounded split exact complex $G_\bullet$ of free $T$-modules.
\end{lemma}

\begin{proof}
Write the differential as $\partial_i:V_i\to V_{i-1}$ and put
\[
B_i=\ker(\partial_i)=\operatorname{Im}(\partial_{i+1}).
\]
At the ends we put $B_{-1}=B_n=0$ after enlarging the indexing interval if necessary.  Exactness gives short exact sequences
\[
0\longrightarrow B_i\longrightarrow V_i
\xrightarrow{\partial_i}B_{i-1}\longrightarrow0.
\]
Since $Q$ is a field, these sequences split.  Choose splittings so that
\begin{equation}\label{eq:split-V}
V_i\cong B_i\oplus B_{i-1}
\end{equation}
and, in these coordinates,
\[
\partial_i(x,y)=(y,0).
\]
For every $i$, choose a basis $\mathcal B_i$ of $B_i$ and let
\[
\widetilde B_i=T^{(\mathcal B_i)}.
\]
Then
\[
Q\otimes_T\widetilde B_i\cong B_i.
\]
Set
\[
G_i=\widetilde B_i\oplus\widetilde B_{i-1}
\]
and define the differential by the same formula
\[
(x,y)\longmapsto(y,0).
\]
The resulting complex $G_\bullet$ is split exact, each $G_i$ is free over $T$, and the chosen bases and splittings in \eqref{eq:split-V} give a chain isomorphism
\[
Q\otimes_TG_\bullet\cong V_\bullet.
\]
\end{proof}

\section{Homological dimensions of torsion modules over the reduced quotient}\label{sec:torsion-comparison}

We now prove the change-of-rings theorem that yields the two dimension formulas.

\begin{theorem}[Torsion comparison theorem]\label{thm:torsion-comparison}
Let $R$ be a strong $\phi$-ring, set $N=\Nil(R)$ and $D=R/N$, and let $M$ be a torsion $D$-module.  Regard $M$ as an $R$-module through the quotient map $R\to D$.  Then
\begin{equation}\label{eq:fd-torsion-equality}
\fd_R(M)=\fd_D(M)
\end{equation}
and
\begin{equation}\label{eq:pd-torsion-equality}
\pd_R(M)=\pd_D(M).
\end{equation}
\end{theorem}

\begin{proof}
We prove the two equalities separately.

\medskip
\noindent\emph{Flat dimension: the inequality $\fd_D(M)\leq\fd_R(M)$.}
Assume first that $\fd_R(M)=n<\infty$.  Choose an $R$-flat resolution
\begin{equation}\label{eq:R-flat-res}
0\longrightarrow F_n\longrightarrow\cdots\longrightarrow F_0
\longrightarrow M\longrightarrow0.
\end{equation}
By Lemma \ref{lem:Tor-N-torsion-zero},
\[
\Tor_i^R(D,M)=0\qquad(i\geq1).
\]
Therefore applying $D\otimes_R-$ to \eqref{eq:R-flat-res} gives an exact sequence
\[
0\longrightarrow D\otimes_RF_n\longrightarrow\cdots
\longrightarrow D\otimes_RF_0\longrightarrow M\longrightarrow0,
\]
where we have used $D\otimes_RM\cong M$.  Each $D\otimes_RF_i$ is flat over $D$.  Hence
\[
\fd_D(M)\leq\fd_R(M).
\]
The same inequality is automatic when $\fd_R(M)=\infty$.

\medskip
\noindent\emph{Flat dimension: the reverse inequality.}
If $\fd_D(M)=\infty$, the preceding inequality already implies $\fd_R(M)=\infty$.  Suppose therefore that
\[
\fd_D(M)=n<\infty
\]
and choose a flat $D$-resolution
\begin{equation}\label{eq:D-flat-res}
0\longrightarrow F_n\longrightarrow\cdots\longrightarrow F_0
\longrightarrow M\longrightarrow0.
\end{equation}
Let $F_\bullet$ denote the complex consisting of the modules $F_n,\ldots,F_0$, without the terminal module $M$.  Its homology is
\[
H_i(F_\bullet)=0\quad(i>0),
\qquad
H_0(F_\bullet)=M.
\]
Put
\[
V_\bullet=Q\otimes_DF_\bullet.
\]
Since $Q$ is flat over $D$ and $M$ is torsion,
\[
Q\otimes_DM=0.
\]
Hence $V_\bullet$ is a bounded exact complex of $Q$-vector spaces.  By Lemma \ref{lem:lift-split-complex}, choose a bounded split exact complex $G_\bullet$ of free $T$-modules and a chain isomorphism
\begin{equation}\label{eq:alpha-complex}
\alpha_\bullet:
Q\otimes_TG_\bullet
\xrightarrow{\ \sim\ }
Q\otimes_DF_\bullet=V_\bullet.
\end{equation}
For every $i$, define the pullback module
\begin{equation}\label{eq:E_i}
E_i
=
G_i\times_{V_i}F_i
=
\left\{(g,f)\in G_i\oplus F_i:
\alpha_i(1\otimes g)=1\otimes f\right\}.
\end{equation}
The canonical square \eqref{eq:canonical-square} is a Milnor square and $T\to Q$ is surjective.  Since $G_i$ is free over $T$ and $F_i$ is flat over $D$, Theorem \ref{thm:patching}(1) shows that $E_i$ is flat over $R$.

The differentials of $G_\bullet$ and $F_\bullet$ preserve the compatibility relation in \eqref{eq:E_i}, because $\alpha_\bullet$ is a chain map.  Thus the $E_i$ form a complex $E_\bullet$.  For every $i$ there is a short exact sequence
\begin{equation}\label{eq:degreewise-MV}
0\longrightarrow E_i
\longrightarrow G_i\oplus F_i
\xrightarrow{\ \delta_i\ }
V_i
\longrightarrow0,
\end{equation}
where
\[
\delta_i(g,f)=\alpha_i(1\otimes g)-1\otimes f.
\]
The map $\delta_i$ is surjective because $T\to Q$ is surjective, so the canonical map $G_i\to Q\otimes_TG_i$ is surjective.  The sequences \eqref{eq:degreewise-MV} are compatible with the differentials and therefore assemble to a short exact sequence of complexes
\begin{equation}\label{eq:complex-MV}
0\longrightarrow E_\bullet
\longrightarrow G_\bullet\oplus F_\bullet
\longrightarrow V_\bullet
\longrightarrow0.
\end{equation}
Both $G_\bullet$ and $V_\bullet$ are exact.  The long exact homology sequence of \eqref{eq:complex-MV} therefore gives
\[
H_i(E_\bullet)\cong H_i(F_\bullet)
\]
for every $i$.  Consequently
\[
H_i(E_\bullet)=0\quad(i>0),
\qquad
H_0(E_\bullet)\cong M.
\]
Thus the canonical augmentation $E_0\twoheadrightarrow H_0(E_\bullet)\cong M$ turns $E_\bullet$ into a flat $R$-resolution of $M$ of length $n$.  Hence
\[
\fd_R(M)\leq n=\fd_D(M).
\]
Together with the first inequality this proves \eqref{eq:fd-torsion-equality}.

\medskip
\noindent\emph{Projective dimension: the inequality $\pd_D(M)\leq\pd_R(M)$.}
If $\pd_R(M)=n<\infty$, apply $D\otimes_R-$ to a projective $R$-resolution of length $n$.  Lemma \ref{lem:Tor-N-torsion-zero} gives exactness after tensoring, and $D\otimes_RP$ is projective over $D$ whenever $P$ is projective over $R$.  Hence
\[
\pd_D(M)\leq\pd_R(M).
\]

\medskip
\noindent\emph{Projective dimension: the reverse inequality.}
If $\pd_D(M)=\infty$, the preceding inequality forces $\pd_R(M)=\infty$.  Assume that
\[
\pd_D(M)=n<\infty
\]
and choose a projective $D$-resolution
\begin{equation}\label{eq:D-proj-res}
0\longrightarrow P_n\longrightarrow\cdots\longrightarrow P_0
\longrightarrow M\longrightarrow0.
\end{equation}
As above, let $P_\bullet$ be the complex formed by $P_n,\ldots,P_0$.  The complex
\[
W_\bullet=Q\otimes_DP_\bullet
\]
is bounded and exact.  Lemma \ref{lem:lift-split-complex} gives a bounded split exact complex $H_\bullet$ of free $T$-modules and a chain isomorphism
\[
\beta_\bullet:Q\otimes_TH_\bullet
\xrightarrow{\ \sim\ }
W_\bullet.
\]
For each $i$ put
\[
L_i=H_i\times_{W_i}P_i.
\]
Both $H_i$ and $P_i$ are projective over $T$ and $D$, respectively.  Milnor patching, Theorem \ref{thm:patching}(2), gives that $L_i$ is projective over $R$.  As in the flat case, the $L_i$ form a complex $L_\bullet$ and there is a short exact sequence of complexes
\[
0\longrightarrow L_\bullet
\longrightarrow H_\bullet\oplus P_\bullet
\longrightarrow W_\bullet
\longrightarrow0.
\]
The two outer complexes $H_\bullet$ and $W_\bullet$ are exact, and hence
\[
H_i(L_\bullet)\cong H_i(P_\bullet).
\]
Thus $L_\bullet$ is a projective $R$-resolution of $M$ of length $n$.  Therefore
\[
\pd_R(M)\leq\pd_D(M).
\]
This proves \eqref{eq:pd-torsion-equality}.
\end{proof}

\begin{corollary}\label{cor:cyclic-dimension-equality}
Let $0\neq J$ be an ideal of $D=R/N$.  Then
\[
\fd_R(D/J)=\fd_D(D/J)
\]
and
\[
\pd_R(D/J)=\pd_D(D/J).
\]
\end{corollary}

\begin{proof}
The $D$-module $D/J$ is torsion: if $0\neq d\in J$, then $d(D/J)=0$.  Apply Theorem \ref{thm:torsion-comparison}.
\end{proof}

\begin{theorem}[Comparison theorem for strong $\phi$-rings]\label{thm:strong-comparison-equalities}
Let $R$ be a strong $\phi$-ring and put $N=\Nil(R)$.  Then
\begin{equation}\label{eq:strong-weak-equality}
\phiwgd(R)=\wgd(R/N)
\end{equation}
and
\begin{equation}\label{eq:strong-global-equality}
\phigd(R)=\gld(R/N).
\end{equation}
\end{theorem}

\begin{proof}
Set $D=R/N$.  If $D$ is a field, then $R$ is $\phi$-von Neumann regular.  Hence $\phiwgd(R)=0$ by Theorem \ref{thm:small-weak}, while $\phigd(R)=0$ by Theorem \ref{thm:small-global}; of course $\wgd(D)=\gld(D)=0$.

Assume that $D$ is not a field.  By Proposition \ref{prop:nonnil-contains-N} and Corollary \ref{cor:cyclic-dimension-equality},
\[
\begin{aligned}
\phiwgd(R)
&=\sup_{0\neq J\subseteq D}\fd_R(D/J)\\
&=\sup_{0\neq J\subseteq D}\fd_D(D/J).
\end{aligned}
\]
The standard cyclic characterization of weak global dimension gives
\[
\sup_{0\neq J\subseteq D}\fd_D(D/J)=\wgd(D).
\]
Here the ideal $J=0$ may be omitted because it contributes only the free module $D$, and a nonfield domain cannot have weak global dimension $0$.  This proves \eqref{eq:strong-weak-equality}.

Similarly, formula \eqref{eq:phi-global-cyclic} and Corollary \ref{cor:cyclic-dimension-equality} give
\[
\begin{aligned}
\phigd(R)
&=\sup_{0\neq J\subseteq D}\pd_R(D/J)\\
&=\sup_{0\neq J\subseteq D}\pd_D(D/J)
=\gld(D),
\end{aligned}
\]
using the standard cyclic characterization of global dimension.  Again $J=0$ may be omitted because $\pd_DD=0$ and a nonfield domain has positive global dimension.  This proves \eqref{eq:strong-global-equality}.
\end{proof}

\begin{remark}\label{rem:questions-ZXQ}
Theorem \ref{thm:strong-comparison-equalities} gives affirmative answers to the two comparison questions of Zhang--Xing--Qi concerning strong $\phi$-rings.  The proof does not require $N^2=0$, a splitting of $R\to R/N$, coherence, or a finiteness hypothesis on $N$.  The essential structural fact is the Milnor square of Theorem \ref{thm:canonical-milnor}.
\end{remark}

We finish with a strongness criterion that uses the kernel identities.

\begin{proposition}\label{prop:R/N-strongly-flat}
Let $R$ be a $\phi$-ring and put $N=\Nil(R)$.  If $R/N$ is strongly $\phi$-flat as an $R$-module, then $R$ is strong.
\end{proposition}

\begin{proof}
Let $a\notin N$ and put
\[
A=(0:_Ra).
\]
Then $A\subseteq K$, so Proposition \ref{prop:K-identities} gives
\[
AN=0.
\]
We compute $\Tor^R_2(R/aR,R/N)$.  From
\[
0\longrightarrow aR\longrightarrow R\longrightarrow R/aR\longrightarrow 0
\]
and $aR\cong R/A$, one obtains the standard identification
\[
\Tor^R_2(R/aR,R/N)
\cong
\Tor^R_1(R/A,R/N)
\cong
\frac{A\cap N}{AN}.
\]
Since $A\subseteq N$ and $AN=0$,
\[
\Tor^R_2(R/aR,R/N)\cong A.
\]
The module $R/aR$ is $\phi$-torsion.  Strong $\phi$-flatness of $R/N$ gives
\[
\Tor^R_2(R/aR,R/N)=0.
\]
Hence $(0:a)=0$ for every $a\notin N$.  Thus every nonnilpotent element is regular, so
\[
Z(R)=N.
\]
\end{proof}

\begin{corollary}\label{cor:all-phiflat-strong}
If every $\phi$-flat $R$-module is strongly $\phi$-flat, then the $\phi$-ring $R$ is strong.
\end{corollary}

\begin{proof}
The module $R/N$ is $\phi$-flat.  Apply Proposition \ref{prop:R/N-strongly-flat}.
\end{proof}

\begin{corollary}\label{cor:nonnil-inj-strong}
Let $R$ be a $\phi$-ring.  If every nonnil-injective $R$-module is strongly nonnil-injective, then $R$ is strong.
\end{corollary}

\begin{proof}
Put $D=R/N$ and let
\[
D^+=\operatorname{Hom}_{\mathbb Z}(D,\mathbb Q/\mathbb Z).
\]
As in \cite{ZhangXingQi}, $D^+$ is nonnil-injective, hence strongly nonnil-injective by hypothesis.  Let $a\notin N$ and $A=(0:a)$.  Character duality gives
\[
\Ext^2_R(R/aR,D^+)
\cong
\operatorname{Hom}_{\mathbb Z}
\bigl(\Tor^R_2(R/aR,D),\mathbb Q/\mathbb Z\bigr).
\]
The calculation in Proposition \ref{prop:R/N-strongly-flat} gives
\[
\Tor^R_2(R/aR,D)\cong A.
\]
Strong nonnil-injectivity forces the Ext group to vanish.  Since $\mathbb Q/\mathbb Z$ is an injective cogenerator of abelian groups, $A=0$.  Hence $(0:a)=0$ for every $a\notin N$, and therefore $Z(R)=N$.
\end{proof}

\end{document}